\documentclass[11 pt]{amsart}

\usepackage{amsmath}
\usepackage{amssymb}
\usepackage{amsfonts}
\usepackage{amsthm}
\usepackage{enumerate}
\usepackage[mathscr]{eucal}
\usepackage[all]{xy}
\usepackage{colortbl}

\usepackage{geometry}
\usepackage{tikz}

\usepackage{bbm}

\usepackage{hyperref}
\usepackage{color}

\usepackage[explicit,indentafter]{titlesec}
\titleformat{\section}{\centering\normalfont\scshape}{\thesection.}{.5em}{#1}
\titleformat{\subsection}[runin]{\normalfont\itshape}{\textnormal{\thesubsection.}}{.5em}{#1.}
\titleformat{\subsubsection}[runin]{\normalfont\itshape}{\thesubsubsection.}{.5em}{#1.}
\titlespacing{\section}{0em}{1em}{0.5em}
\titlespacing{\subsection}{0em}{.5em}{0.5em}

\hypersetup{
	colorlinks=true,
	citecolor=[rgb]{0,0,0.5},
	linkcolor=[rgb]{0,0,0.5},
	urlcolor=[rgb]{0,0,0.75},
	pdfpagemode=UseNone,
	pdfstartview=FitH,
	pdfdisplaydoctitle=true,
	pdftitle={Bounds for spherical maximal operators},
	pdfauthor={T. C. Anderson, K. Hughes, J. Roos, A. Seeger},
	pdflang=en-US
}

\newtheorem{thm}{Theorem}
\newtheorem{thms}{Theorem}[section]
\newtheorem{lemma}[thms]{Lemma}
\newtheorem{corollary}[thms]{Corollary}
\newtheorem{proposition}[thms]{Proposition}

\numberwithin{equation}{section}

\theoremstyle{remark}
\newtheorem{remark}[thms]{Remark}

\newcommand{\Be}{\begin{equation}}
\newcommand{\Ee}{\end{equation}}
\newcommand{\Bea}{\begin{align}}
\newcommand{\Eea}{\end{align}}
\newcommand{\Beas}{\begin{align*}}
\newcommand{\Eeas}{\end{endalign*}}
\newcommand{\Benu}{\begin{enumerate}}
\newcommand{\Eenu}{\end{enumerate}}
\newcommand{\Bi}{\begin{itemize}}
\newcommand{\Ei}{\end{itemize}}

\def\bbone{{\mathbbm 1}}

\def\bd{\mathrm{bd}}

\def\Zj{{\mathcal Z_j}}
\def\dimthE{\dim_{A,\theta}\!E}

\def\intslash{\rlap{\kern  .32em $\mspace {.5mu}\backslash$ }\int}
\def\qsl{{\rlap{\kern  .32em $\mspace {.5mu}\backslash$ }\int_{Q_x}}}

\def\vth{\vartheta}

\def\Q{\mathcal Q}

\def\emph#1{{\it #1 }}

\def\ga{\gamma}

\def\inn#1#2{\langle#1,#2\rangle}

\def\noi{\noindent}

\def\meas{{\text{\rm meas}}}

\def\lc{\lesssim}
\def\gc{\gtrsim}

\def\eps{\varepsilon}
\def\ep{\epsilon}

\def\ka{\kappa}
            
\def\la{\lambda}

\def\om{\omega}

\def\fM{{\mathfrak {M}}}

\def\fS{{\mathfrak {S}}}

\def\bbC{{\mathbb {C}}}

\def\bbN{{\mathbb {N}}}

\def\bbR{{\mathbb {R}}}

\def\bbZ{{\mathbb {Z}}}

\def\cC{{\mathcal {C}}}

\def\cI{{\mathcal {I}}}
\def\cJ{{\mathcal {J}}}

\def\cM{{\mathcal {M}}}

\def\cQ{{\mathcal {Q}}}
\def\cR{{\mathcal {R}}}
\def\cS{{\mathcal {S}}}
\def\cT{{\mathcal {T}}}

\def\cV{{\mathcal {V}}}

\def\cZ{{\mathcal {Z}}}

\def\Q{{\hbox{\bf Q}}}

\def\be#1{\begin{equation}\label{#1}}
\def\endeq{\end{equation}}
\def\endal{\end{align}}
\def\bas{\begin{align*}}
\def\eas{\end{align*}}
\def\bi{\begin{itemize}}
\def\ei{\end{itemize}}

\newcommand{\Qfourx}[1]{\d*(\d-1)/(\d*\d+2*#1-1)}
\newcommand{\Qfoury}[1]{(\d-1)/(\d*\d+2*#1-1)}
\def\Qthreex{(\d-\b)/( \d-\b+1)}
\def\Qthreey{1/(\d-\b+1)}

\newcommand{\definecoords}{
    \def\ptsize{.1pt}
	\coordinate (Q1) at (0,0);
	\coordinate (Q2) at ( {(\d-1)/(\d-1+\b)}, {(\d-1)/(\d-1+\b)}  );
	\coordinate (Q3) at ( {\Qthreex},  { \Qthreey }  );
	\coordinate (Q4b) at ( { \Qfourx{\b}  }, { \Qfoury{\b} }  );
	\coordinate (Q4g) at ( { \Qfourx{\g}  }, { \Qfoury{\g} }  );
	\coordinate (R) at ( { \d*(\d-1)/ (\d*\d-1+\b) } , { (\d-1)/(\d*\d-1+\b) } );
	\coordinate (C1) at ( { (\Qfourx{\b}+\Qfourx{\g})/2 }, {(\Qfoury{\b}+\Qfoury{\g})/2}  ); 
	\coordinate (C2) at (Q4b);
}

\newcommand{\drawauxlines}[2]{ 
	\draw (0,0) [->] -- (0,1) node [left] {$\frac1q$};
	\draw (0,0) [->] -- (1,0) node [below] {$\frac1p$};
	\draw [dashed,opacity=.3] (1,0) -- (0,1);
	\draw [dashed,opacity=.3] (#1) -- (1,{1/\d}); 
	\draw [dashed,opacity=.3] (#2) -- (1,1); 
	\draw [dashed,opacity=.1]
	(.5, 0) -- (.5, .5);
}

\newcommand{\drawQbg}{
	\fill (Q1) node [left] {$Q_1$} circle [radius=.02em];
	\fill (Q2) node [right] {$Q_{2,\beta}$} circle [radius=\ptsize];
	\fill (Q3) node [right] {$Q_{3,\beta}$} circle [radius=\ptsize];
	\fill (Q4g) node [below] {$Q_{4,\gamma}$} circle [radius=\ptsize];
	\fill [opacity=.2] (Q1) -- (Q2) -- (Q3) -- (Q4g) -- cycle;
	\draw [opacity=.6] (Q1) -- (Q2) -- (Q3) -- (Q4g) -- cycle;
}	

\newcommand{\arxiv}[1]{\href{https://www.arxiv.org/abs/#1}{arXiv:#1}}

\begin{document}
\title
[Bounds for spherical maximal operators]
{$\mathbf{L^p\to L^q}$ bounds 
for \\spherical maximal operators}
\author
[T. Anderson \ \ K. Hughes \ \ J. Roos \ \ A. Seeger]
{Theresa C. Anderson \  \ \ \ Kevin Hughes \\ \\ Joris Roos \  \ \ \ Andreas Seeger\\}
\subjclass[2010]{}
\thanks{Research supported in part by the National Science Foundation}

\address{Theresa C. Anderson \\ Department of Mathematics \\Purdue University \\150 N University St, West Lafayette, IN 47907, USA} \email{tcanderson@purdue.edu}

\address{Kevin Hughes \\
School of Mathematics, The University of Bristol, Fry Building, Woodland Road, BRISTOL, BS8 1UG, UK, and the Heilbronn Institute for Mathematical Research, Bristol, UK}
\email{khughes.math@gmail.com} 

\address{Joris Roos \\ Department of Mathematics \\ University of Wisconsin \\480 Lincoln Drive\\ Madison, WI
53706, USA} \email{jroos@math.wisc.edu}

\address{Andreas Seeger \\ Department of Mathematics \\ University of Wisconsin \\480 Lincoln Drive\\ Madison, WI
53706, USA} \email{seeger@math.wisc.edu}

\begin{abstract}
Let $f\in L^p(\mathbb R^d)$, $d\ge 3$,  and let $A_t f(x)$ be the average of $f$ over the sphere with radius $t$ centered at $x$. 
For a subset $E$ of $[1,2]$ we prove close to sharp 
$L^p\to L^q$  estimates for the maximal function $\sup_{t\in E} |A_t f|$. A new feature is the dependence  of the results on both the upper Minkowski dimension of $E$ and the Assouad dimension of $E$. The result can be  applied to prove sparse domination bounds  for a related global spherical maximal function.
\end{abstract}


\maketitle 

\section{Introduction and statement of results}
Let $A_t f(x)$ denote the mean of a locally integrable function $f$ over the sphere with radius $t$ centered at $x$. That is,
\[
A_t f(x)= \int f(x-ty) d\sigma(y),
\]
where $\sigma$ is the standard normalized surface measure on the unit sphere in $\bbR^d$ and $d\ge 2$.
Let $E\subset [1,2]$ and 
\begin{equation} \label{MEdef}M_E f(x) =\sup_{t\in E}|A_t f(x)|, \end{equation} which is well defined as a measurable function at least for  continuous $f$. We consider the problem of {\it $L^p$-improving} estimates, i.e. $L^p\to L^q$ estimates for $q>p$, partially motivated by the problem of  sparse domination results for the    global maximal function 
$\fM_E f(x) =\sup_{k\in \bbZ} \sup_{t\in E}|A_{2^kt} f(x)|$,  dependent on the geometry of $E$, see \S\ref{sparsesect}.  The sparse domination problem is  suggested by a remark in \cite{lacey}.

It is well known (\cite{littman}) that for $E=\{\text{point}\}$ (when $M_E$ reduces to a single average) we have $L^p\to L^q$ boundedness if and only if $(1/p,1/q)$ belongs to the closed triangle  with corners 
$(0,0)$, $(1,1)$ and $(\frac{d}{d+1}, \frac{1}{d+1})$.
For the other extreme case $E=[1,2]$ a necessary condition for $L^p\to L^q$ boundedness is that
$(1/p, 1/q)$ belongs to the closed quadrangle $\cQ$ with corners
$P_1=(0,0)$, $P_2=(\frac{d-1}{d},\frac{d-1}d)$,
$P_3=(\frac{d-1}{d}, \frac{1}{d})$ and $P_4=(\frac{d(d-1)}{d^2+1}, \frac{d-1}{d^2+1})$, see \cite{schlag-sogge}. 
By results of Stein \cite{stein} for $d\ge 3$, and Bourgain \cite{bourgain2} there is  a positive result for the segment $[P_1,P_2)$ while boundedness fails at $P_2$.
For $p<q$ 
almost sharp results are due to Schlag and Sogge \cite{schlag-sogge} (see also previous work by Schlag \cite{schlag} on the circular maximal function) and additional endpoint results were obtained by Lee \cite{slee}. For the point $P_2$ Bourgain \cite{bourgain1} had shown a restricted weak type inequality, and Lee \cite{slee} also showed in addition a restricted weak type inequality for the 
points $P_3$ and $P_4$. This implies by interpolation that $M_{[1,2]}$ satisfies strong type bounds on the half-open edge $[P_1, P_2)$ and the open edges $(P_1, P_4)$, $(P_4, P_3)$. Moreover,  restricted strong type estimates hold on the half-open edge $[P_2, P_3)$. It is not known whether the $L^p\to L^q$  bound holds for $P_3$ or  $P_4$. 
In two dimensions  the quadrangle $\cQ$ becomes a triangle as the points $P_2$ and $P_3$ coincide. From \cite{slee} we have that $L^p\to L^q$ boundedness holds 
on $\cQ$ with exception of the points $P_2=P_3$ and $P_4$. Lee also shows the $L^{5/2,1}(\bbR^2)\to L^{5,\infty}(\bbR^2)$ estimate, i.e. the restricted weak type inequality corresponding to $P_4$ (and it is open whether the endpoint $L^{5/2}\to L^5$ estimate holds). 
In two dimensions, for the point $P_2=P_3$  the endpoint  restricted weak type inequality is true for radial functions (\cite{leckband}) but fails for general functions, see \S8.3 of \cite{stw}.

In this paper we take up the case of $L^p$ improving estimates for spherical maximal functions with sets of dilations intermediate  between the two above extreme cases;
here we mainly  consider the problem in dimensions $d\ge 3$ although some partial results in two dimensions are included.  Satisfactory results for $p=q$  are in 
\cite{sww1} where it was shown that the precise range of $L^p$ boundedness depends on the upper Minkowski dimension $\beta$ of the set $E$, which should also play a role for $L^p\to L^q$ estimates.
However it turns out that the notion of  upper Minkowski dimension alone is not appropriate  to determine the range of $L^p\to L^q$ boundedness, and that in addition another type of dimension, the  upper Assouad dimension, plays a significant role.

We recall the definitions.
For a set $E\subset \bbR$ and $\delta>0$ denote by  $N(E,\delta)$  the minimal number of compact intervals of length $\delta$ needed to cover $E$. The
{\it upper Minkowski dimension $\dim_M\!E$}  of a compact set $E$ is the smallest $\beta$ so that there is an estimate \begin{equation}\label{minkdim}N(E,\delta)\le C(\eps)\delta^{-\beta-\eps} \end{equation} for all $\delta<1$ and $\eps>0$.
The {\it upper Assouad dimension $\dim_A\!E$} 
is the smallest number $\gamma$ so that there exist $\delta_0>0$, and constants $C_\eps$ for all $\eps>0$  such that for all $\delta\in (0,\delta_0)$ and all
intervals  $I$ of length $|I| \in (\delta,\delta_0)$ we have 
\begin{equation} \label{densgamma}N(E\cap I, \delta) \le C_\eps (\delta/|I|)^{-\gamma-\eps}\,.\end{equation} 
Clearly we have $0\le  \dim_M\!E\le  \dim_A\!E 
\le 1$  for every compact subset of $\bbR$.
For the Cantor middle third set  $C$   we have 
 $\dim_M\!C=\dim_A\!C=\log_32$. More generally the upper Minkowski and upper Assouad dimensions are equal for large classes of quasi-self-similar sets, see \cite[\S2.2]{fraser2014} for precise definitions. 
 In contrast, if $0<\beta<1$ then  for the set $E(\beta)=\{1+n^{-a(\beta)}: n\in \bbN\}$, with $a(\beta)=\frac{1-\beta}{\beta}$
  we have 
 $\dim_M \!E(\beta)=\beta$ and $\dim_A\!E(\beta)=1$. 
 
One seeks to determine the  region of $(1/p,1/q)$ for which $\|M_E\|_{L^p\to L^q}$ is finite. It turns out that the following definitions are relevant to answer this question, up to endpoints.

\medskip 
\noi{\it Definition.} (i)  For $\beta\le \gamma\le 1$ let $\cQ(\beta,\gamma)$ be the closed convex hull of the points
\begin{equation}\label{Qbetagamma}\begin{gathered}
Q_1=(0,0),\qquad Q_2(\beta)=(\tfrac{d-1}{d-1+\beta},\tfrac{d-1}{d-1+\beta}),
 \\Q_3(\beta)=(\tfrac{d-\beta}{d-\beta+1}, \tfrac{1}{d-\beta+1} ), \qquad
Q_4(\gamma)= (\tfrac{d(d-1)}{d^2+2\gamma-1}, \tfrac{d-1}{d^2+2\gamma-1}).
\end{gathered}\end{equation}
\noi (ii) Let  $\text{Seg}(\beta)$ be  the  line segment connecting $(0,0)$ and $Q_2(\beta) $, with $(0,0)$ included and $Q_2(\beta)$ excluded. 

\noi(iii) Let $\cR(\beta,\gamma)$ denote the union of $\mathrm{Seg}(\beta)$ and the interior of $\mathcal{Q}(\beta,\gamma)$.

\medskip

Note that $$\cR(\beta,\gamma_2)\subsetneq  \cR(\beta,\gamma_1) \text{ if $\beta\le \gamma_1<\gamma_2\le 1$.}
$$

\begin{figure}[ht]
\begin{tikzpicture}[scale=4.25]
\def\d{3}
\def\b{.8}
\def\g{1}
\definecoords
\drawauxlines{Q4g}{Q2}
\drawQbg
\end{tikzpicture}
\caption{The region $\mathcal{Q}(\beta,\gamma)$ with $d=3$, $\beta=0.8$, $\gamma=1$.}\label{quadrangle}
\end{figure}
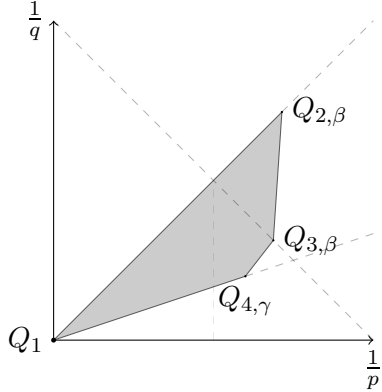

It was shown in \cite{sww1} that boundedness holds on the segment $\text{Seg}(\beta)$ and this is sharp up to the endpoint.  A number of conjectures for  endpoint situations for $L^p\to L^p$ boundedness are in 
\cite{sww2} and these conjectures were confirmed there for the problem of $L^p\to L^p$ estimates on {\it radial} functions; see also \cite{stw} for partial results for convex sequences when the radiality assumption can be dropped.
A slight variation of the arguments in \cite{sww1} shows that in the interior of the triangle with corners $Q_i(\beta)$, $i=1,2,3$ we have $L^p\to L^q$ boundedness, see \S\ref{prelimsect}. Interpolation with the above mentioned results by Schlag-Sogge and Lee then shows that we have $L^p\to L^q$ boundedness in the region $\cR(\beta,1)$. On the other hand the standard examples (cf. \S\ref{sec:lowerbd1}, \S\ref{sec:lowerbd2}, \S\ref{sec:lowerbd3}) show that boundedness fails in the complement of $\cQ(\beta,\beta)$.  The main  result of this paper is to close this gap (at least in dimensions $d\ge 3$).

\begin{thm}\label{pqthm} Let $d\ge 3$, $0\le \beta\le \gamma\le 1$ or $d=2$, $0\le \beta\le\gamma\le 1/2$. 
Let $E$ be a subset of $[1,2]$  with
$\dim_M\!E=\beta$, 
$\dim_A\!E=\gamma$. Then for $(1/p,1/q)$ contained in $\cR(\beta,\gamma)$,
\begin{equation}\label{MEpq}\|\sup_{t\in E} |A_t f | \|_q \lc \|f\|_p.\end{equation}
\end{thm}

\noi {\it Remark.} The conclusion of the theorem in the two-dimensional case continues to hold in the case  $\gamma>1/2$ not covered in this paper. This requires arguments different from what we use here, see \cite{rs}. 

\smallskip

We now turn to the issue of sharpness. It turns out that Theorem \ref{pqthm} is sharp up to endpoints for a large class of sets which includes the above mentioned convex sequences 
$E_a=\{1+ n^{-a}\}$ where $\dim_M\! E_a= (a+1)^{-1}$ and $\dim_A \!E_a= 1$, and also sets with $\dim_A \!E= \dim_M \!E$ (in particular, self--similar sets). Moreover we shall, for every $\beta\le \gamma\le 1$,   construct sets $E(\beta, \gamma)$  with $\dim_M \!E(\beta, \gamma)=\beta$,
$\dim_A \!E(\beta, \gamma)=\gamma$ so that Theorem \ref{pqthm} is sharp up to endpoints for these sets, meaning that $L^p\to L^q$ boundedness of $M_E$ fails if $(1/p, 1/q)\notin  {\cQ(\beta, \gamma)}$. 

We can say more about the  sets $E$ for which such sharpness results can be proved.  To describe this family 
we work with  definitions of dimensions which interpolate between upper Minkowski dimension and Assouad dimension,  notions that were introduced by Fraser and Yu in \cite{fraser-yu1}.
For $0\le \theta<1$ one  defines $\dim_{A,\theta}\!E$ 
to be the smallest number  $\gamma(\theta)$ so that there exist $\delta_0>0$, and constants $C_\eps$ for all $\eps>0$  such that for all $\delta\in(0,\delta_0)$ and all intervals  $I$ of length $|I|=\delta^\theta$ we have 
\begin{equation} \label{densgammatheta} N(E\cap I, \delta) \le C_\eps (\delta/|I|)^{-\gamma(\theta)-\eps}
\,.\end{equation} 
The function $\theta\mapsto \dimthE$ is called the {\it Assouad spectrum} of $E$. 
Note that $\dim_{A,0}\!E=\dim_M\!E$.
There are some immediate inequalities  relating the Assouad 
spectrum with Minkowski and Assouad dimensions (see \cite[Prop. 3.1]{fraser-yu1}),
\begin{equation} \label{MinkversAssouadspec}
\dim_M\!E\le\dimthE\le
\mathrm{min}\big(\tfrac{\dim_M\,E}{1-\theta}, \dim_A\!E\big). \end{equation}

Indeed, the inequality $\dimthE\le \dim_A\!E$ holds by definition, while the inequality $\dimthE\le \dim_M\!E /(1-\theta)$ follows from $N(E\cap I,\delta)\le N(E,\delta)$.
To see the first inequality  in \eqref{MinkversAssouadspec} let us write $\beta=\dim_M \!E$ and $\gamma(\theta)=\dimthE$. Cover the set $E$  with an essentially disjoint collection $\cI$ of intervals $I$ of length $\delta^{\theta}$ so that
$\# \cI \le 2N(E,\delta^\theta)\le C(\eps_1) (\delta^{\theta})^{-\beta-\eps_1}$
and use
\begin{align*} &N(E,\delta)\le \sum_{I\in \cI} 
N(E\cap I, \delta) \le \sum_{I\in \cI} C_\eps(\delta^{\theta-1} )^{\gamma(\theta) +\eps}\\
&\le C_\varepsilon C(\varepsilon_1) \delta^{-\theta(\beta+\varepsilon_1)-(1-\theta)(\gamma(\theta)+\varepsilon)}.
\end{align*}
By definition of Minkowski dimension and letting $\varepsilon$, $\varepsilon_1$ tend to zero, we get
$\beta\le \theta \beta  + (1-\theta) \gamma(\theta)$, which implies $\beta\le \gamma(\theta)$ since $0\le \theta<1$. For more sophisticated relations between the various dimensions in the Assouad spectrum, 
see \cite{fraser-yu1}. The papers 
\cite{fraser-yu1}, \cite{fraser-yu2} contain discussions of many interesting examples that are relevant  in the context of Assouad dimension and Assouad spectrum.

Here we are interested, for suitable sets $E$,  in those values of $\theta$ for which 
\begin{equation}\label{constancyofAssouadspectr}
\dimthE=\dim_A\!E.\end{equation}
While the Assouad spectrum is generally not monotone (see \cite[\S 8]{fraser-yu1}), it holds that once the Assouad spectrum reaches the Assouad dimension then it stays there, {\it i.e.}  if $\dim_{A,\theta_0} \!E=\dim_A\!E$ then 
$\dim_{A,\theta} \!E=\dim_A\!E$ for $\theta_0<\theta<1$ (see \cite[Cor. 3.6]{fraser-yu1}). Note that the upper bound in \eqref{MinkversAssouadspec} implies that \eqref{constancyofAssouadspectr} can only hold for $\theta\ge 1-\beta/\gamma$, where $\beta=\dim_M\!E$ and $\gamma=\dim_A\!E$.
This leads us to introduce the following terminology.

\medskip 

\noi{\it Definition.} 
We say that a set $E$ is {\it $(\beta,\gamma)$-Assouad regular} if
$\dim_M\!E=\beta$, $\dim_A\!E=\gamma$ and
$\dimthE=\dim_A\!E \text { for $1>\theta>1-\beta/\gamma$.}$  
$E$ is called {\it Assouad regular} if it is {$(\beta,\gamma)$-Assouad regular} for some pair $(\beta,\gamma)$.

\medskip

Note that when $\dim_M\!E=\dim_A\!E $ or $\dim_M\!E=0$, then $E$ is always Assouad regular. Also, the convex sequences $E_a=\{1+ n^{-a}\}$ are $(\frac{1}{a+1},1)$-Assouad regular (see \cite[Thm. 6.2]{fraser-yu1}). In \S\ref{Assreg} we shall give examples of $(\beta,\gamma)$-Assouad regular sets, for every pair $(\beta,\gamma)$ with $0<\beta< \gamma\le 1$. We shall show that Theorem \ref{pqthm} is sharp up to endpoints for Assouad regular sets.

\begin{thm}\label{thm:lowerbd}
	Let $d\ge 2$, $E\subset [1,2]$ and $\beta = \dim_M\!E$.
	\begin{enumerate}
	\item[(i)] If $(1/p,1/q)\not\in {\cQ(\beta, \beta)}$, then 
	\begin{equation}\label{eqn:lowerbdcon}
	\sup\{\|M_E f\|_q: \|f\|_p\le 1 \}=\infty.
	\end{equation}

	\item[(ii)] Let $\theta\in [0,1)$ such that
	\begin{equation}\label{eqn:lowerbdcond}
	\dimthE = \tfrac{\dim_M\!E}{1-\theta}.
	\end{equation}
	Then \eqref{eqn:lowerbdcon} holds for $(1/p,1/q)\notin {\cQ(\beta, \tfrac{\beta}{1-\theta})}$.
	
	\item[(iii)] If $0\le\beta\le \gamma\le 1$ and $E$ is $(\beta,\gamma)$--Assouad regular, then \eqref{eqn:lowerbdcon} holds for $(1/p,1/q)\not\in {\cQ(\beta, \gamma)}$. In particular, Theorem \ref{pqthm} is sharp up to endpoints for Assouad regular sets.
	\end{enumerate}
\end{thm}

Observe that (ii) implies (i) because \eqref{eqn:lowerbdcond} holds trivially for $\theta=0$. Moreover, if $E$ is $(\beta,\gamma)$--Assouad regular, then \eqref{eqn:lowerbdcond} holds with 
$\theta=1-\beta/\gamma$, 
i.e. $\gamma=\tfrac{\beta}{1-\theta}$, 
so (ii) also implies (iii). The validity of (ii) is proven in \S \ref{sec:lowerbd}.

It would be interesting to investigate  the sharpness of Theorem \ref{pqthm} for sets $E$ which are not Assouad regular. For more on this topic, see \cite{rs}.

\subsection*{Endpoint results} Here we discuss endpoint questions on the off-diagonal boundaries of $\cQ(\beta,\gamma)$ and give a result which is somewhat analogous to one of Lee's theorems in \cite{slee}. The theorem involves restricted weak type estimates 
(with  Lorentz spaces $L^{p,1}$, $L^{q,\infty}$) at the points $Q_2(\beta)$, $Q_3(\beta)$ and $Q_4(\gamma)$ and strong type estimates on the open edges connecting these points. Recall that $M_E$ is said to be of strong type $(p,q)$ if $M_E:L^p\to L^q$ is bounded, and of restricted weak type $(p,q)$ if $M_E: L^{p,1}\to L^{q,\infty} $ is bounded. To prove these results we need to slightly strengthen the dimensional assumptions in Theorem \ref{pqthm}.

\begin{thm}\label{pqthmendpt} Let $d\ge 3$,  $0\le \beta\le \gamma\le 1$, or $d=2$, $0\le \beta\le\gamma< 1/2$.  Let $E\subset[1,2]$.

\begin{enumerate}
\item[(i)] Suppose that 
\begin{equation}\label{endptminkassu}
\sup_{0<\delta<1} \delta^\beta N(E,\delta)<\infty.\end{equation}
If $(1/p,1/q)$ is one of the points
\[ Q_2(\beta)= 
(\tfrac{d-1}{d-1+\beta}, \tfrac{d-1}{d-1+\beta}),\,  
Q_3(\beta)=(\tfrac{d-\beta}{d-\beta+1}, \tfrac{1}{d-\beta+1})\] then  $M_E$ is of restricted weak type $(p,q)$. If in addition  $\beta<1$ then 
$M_E$ is of strong type $(p,q)$ whenever $(\frac 1p,\frac 1q)$ belongs to the open line segment connecting $Q_2(\beta)$ and $Q_3(\beta)$.
\item[(ii)] Suppose that 
\begin{equation}\label{endptassouadassu}
\sup_{0<\delta<1} \sup_{\delta\le |I|\le 1}\big(\tfrac{\delta}{|I|}\big)^\gamma  N(E\cap I,\delta)<\infty,
\end{equation}
where  the second supremum is taken over all intervals $I$ of length in $[\delta,1]$. 
  Let $(\frac 1p,\frac 1q)= Q_4(\gamma)=  (\frac{d(d-1)}{d^2+2\gamma-1}, \frac{d-1}{d^2+2\gamma-1})$.
 Then $M_E$  is of restricted weak type $(p,q)$.
 \item[(iii)] 
Suppose that $\beta<1$ and  that both \eqref{endptminkassu} {and} \eqref{endptassouadassu} hold. Then $M_E$ is of strong type $(p,q)$ for all $(\frac 1p, \frac 1q)\in {\cQ(\beta,\gamma)}\setminus 
 \{ Q_2(\beta), Q_3(\beta), Q_4(\gamma)\} $.
\end{enumerate}
\end{thm}

\subsection*{This paper} In \S\ref{prelimsect} we begin proving Theorems \ref{pqthm} and \ref{pqthmendpt} by discussing elementary and basically known estimates relevant for the $p=q$ cases and the bounds at  $Q_3(\beta)$.
In \S\ref{assouadsect} we prove the upper bounds at $Q_4(\gamma)$, thus concluding the proofs of Theorems \ref{pqthm} and \ref{pqthmendpt}.
In \S\ref{sec:lowerbd} we discuss examples proving Theorem 
\ref{thm:lowerbd}; see \S\ref{sec:lowerbd4} for the new argument of sharpness for Assouad regular sets.
In \S\ref{Assreg} we give some relevant  constructions of sets with prescribed  Minkowski and Assouad dimensions.
\S\ref{sparsesect} contains a discussion of related sparse domination bounds for the global maximal operator $\mathfrak{M}_E$.

\section{Preliminary results}\label{prelimsect}
In this section we assume $d\ge 2$. We dyadically decompose the multiplier of the spherical means. Let 
$\eta_0$ be a $C^\infty$ function with compact support in $\{\xi:|\xi|< 2\}$ such that $\eta_0(\xi)=1$ for $|\xi|\le 3/2$. For $j\ge 1$ set $\eta_j(\xi)= \eta_0(2^{-j}\xi)- \eta_0(2^{1-j}\xi)$ so that $\eta_j$ is supported in the annulus $\{\xi: 2^{j-1}<|\xi|<2^{j+1}\}$. Let  $\sigma$ denote the surface measure of the unit sphere in $\bbR^d$.
Define $A^j_t f$, $j=0,1,2,\dots$  via the Fourier transform by
\begin{equation}\label{Ajtdef} 
\widehat {A_t^j f}(\xi) = \eta_j(\xi) \widehat \sigma(t\xi) 
\widehat f(\xi).
\end{equation}

We change notation for added flexibility.
Let $a(t,\cdot)$ be a multiplier and a symbol  of order zero, satisfying
$|\partial_t^M \partial_\xi^\alpha a(t,\xi)|\le C |\xi|^{-\alpha}$  for all multiindices $\alpha$ with $|\alpha|\le 100 d$ and all $M$. 
Denote by $\fS_0$ the class of these symbols.
For $a\in \fS_0$ and $j\ge 1$ let
$$T^{\pm, j}_t [a,f](x) = \int \eta_j(\xi) a(t,\xi)  \widehat f(\xi)e^{i\inn{x}{\xi}\pm i t|\xi|}
 d\xi$$
so that, by  well-known stationary phase arguments (see \cite[Ch. VIII]{Steinbook3}),  $$A^j_t f= 2^{-j(d-1)/2} ( T_t^{+,j}[a_{j,+ },f]+T_t^{-,j}[a_{j,- },f]),$$ 
where $a_{j,\pm}$ are symbols in $\fS_0$, with bounds uniform in $j$. In what follows $a_j\in \fS_0$ is fixed and $T^j_t$ refers to either $f\mapsto T_t^{\pm,j}[a_{j,\pm },f]$. 

We shall need a pointwise estimate for the convolution kernels of the operators $T^j_t$ and $T^j_t (T^j_{t'})^*$ provided by the following lemma.
\begin{lemma} \label{ptwlem}
Let $\chi\in C^\infty_c(\bbR^d)$, supported in $\{\xi: 1/2<|\xi|\le 2\}$ and let
\[\ka^{j,\pm} (x,t)= \int  \chi(2^{-j} \xi) e^{i\langle x,\xi\rangle\pm i t|\xi|} d\xi.
\]
Then there are constants $C_N$ depending only on bounds for a finite number of derivatives of $\chi$ so that  for all $(x,t)\in\bbR^d\times \bbR$: 
\begin{equation} \label{ptwb}|\ka^{j,\pm} (x,t)| 
\le C_N 2^{jd} (1+2^j|x|)^{-\frac{d-1}{2} } (1+2^j\big| |x|-|t|\big|)^{-N}.
\end{equation}
\end{lemma}
\begin{proof} We change variables and write
\[\ka^{j,\pm} (x,t)=2^{jd} 
 \int  \chi( \om) e^{i2^j\langle x,\om\rangle\pm i t2^j|\om|} d\om.
\]
If $\max \{|x|,  |t|\} \le C 2^{-j}$  we use the trivial estimate 
$|\ka^{j,\pm} (x,t)| \le 2^{jd} $.
From integration by parts we obtain
\[
|\ka^{j,\pm} (x,t)| \lc_M \begin{cases} 
 2^{jd}  (1+2^j|x|)^{-M} &\text{ if } |x|>2|t|,
\\
 2^{jd}  (1+2^j|t|)^{-M} &\text{ if } |t|>2|x|.
\end{cases}
\]
It remains to consider the case $|t|\approx |x|>2^{-j}$.
Then we apply polar coordinates,  stationary phase in the spherical variables, and integration by parts in the resulting oscillatory integral to get \eqref{ptwb}.
\end{proof}

We now state the basic estimate used in \cite{sww1}.

\begin{lemma}\label{pp-lemma}
(i)  For $1\le p\le 2,$
$$2^{-j(d-1)/2} \|\sup_{t\in E} | T^j_t f| \|_p\lc N(E, 2^{-j} )^{1/p} 2^{-j(d-1)(1-1/p)}\|f\|_p.$$
(ii) For $2\le p\le\infty$,
$$2^{-j(d-1)/2} \|\sup_{t\in E} | T^j_t f| \|_p\lc N(E, 2^{-j} )^{1/p} 2^{-j(d-1)/p)}\|f\|_p.$$
\end{lemma} 
\begin{proof} For (i) one interpolates between the cases $p=1$ and $p=2$,
and 
for (ii) one interpolates between the cases $p=\infty$ and $p=2$.
\end{proof}

The same argument also gives

\begin{lemma} \label{pq-lemma} 
For $2\le q\le\infty$, $1/q'+1/q=1$,
$$2^{-j(d-1)/2} \|\sup_{t\in E} | T^j_t f| \|_{q} \lc N(E, 2^{-j} )^{\frac 1q} 2^{j(1-\frac{d+1}{q})}\|f\|_{q'}.$$
\end{lemma} 
\begin{proof}We interpolate between $q=2$ and $q=\infty$. The case  $q=2$ is from the previous lemma. For the case  $q=\infty$ we use that the convolution kernel $K^j_t$ of $2^{-j(d-1)/2} T^j_t$ satisfies 
the uniform bound $|K_t^j(x)|\lesssim 2^j$ (by Lemma \ref{ptwlem}).
\end{proof}

\subsection*{Bourgain's interpolation trick} For various restricted weak type estimates we apply a familiar interpolation argument due to Bourgain \cite{bourgain1}, see also an abstract extension in the appendix of \cite{carberyetal}. It says assuming $a_0, a_1>0$,  that if $(R_j)_{j\ge 0}$ are sublinear operators which map $L^{p_0,1}$ to $L^{q_0,\infty} $ with operator norm
$O(2^{ja_0})$ 
and 
$L^{p_1,1}$ to $L^{q_1,\infty} $ with operator norm
$O(2^{-ja_1})$ then $\sum_{j\ge 0} R_j$ is of restricted weak type $(p,q)$ where 
$$(\tfrac 1p, \tfrac 1q)=(1-\theta) 
(\tfrac 1{p_0}, \tfrac 1{q_0})+\theta
(\tfrac 1{p_1}, \tfrac 1{q_1}), \quad \theta=\tfrac{a_0}{a_0+a_1}.$$

Using this result we get

\begin{lemma} \label{pqrwtlemma}
Suppose $0<\beta<1$ and assumption \eqref{endptminkassu} holds. 
Then $M_E$ is of restricted weak type $p,q$ if 
 $(1/p,1/q)$  is either one of $Q_2(\beta)$, $Q_3(\beta)$.
\end{lemma} 

\begin{proof}
For the statement with $Q_2(\beta)$ we apply Lemma \ref{pp-lemma} and assumption \eqref{endptminkassu} 
to get for $1\le p\le 2$,
\[ \|\sup_{t\in E} |A^j_t f| \|_p \lc 2^{j (\frac {d-1+\beta}{p}-d+1)} \|f\|_p. \]
We consider these inequalities for $p_0$, $ p_1$ where  $p_0< \frac{d-1+\beta}{d-1}<p_1$. 
We then use Bourgain's
interpolation argument to deduce
\[ \Big\|\sum_{j\ge 0} \sup_{t\in E} |A^j_t f| \Big\|_{L^{p,\infty}} \lc \|f\|_{L^{p,1}}, \quad p= \tfrac{d-1+\beta}{d-1}. \]
This gives  the asserted weak restricted weak type inequality  for $M_E$ at $Q_2(\beta)$.

For the result at $Q_3(\beta)$ we apply Lemma \ref{pq-lemma} instead and obtain under assumption \eqref{endptminkassu}, for $2\le q\le\infty$, 
\[ \|\sup_{t\in E} |A^j_t f| \|_q \lc 2^{j (1-\frac{d+1-\beta}{q})}  \|f\|_{q'}. \]
Bourgain's interpolation argument gives
\[ \Big\|\sum_{j\ge 0} \sup_{t\in E} |A^j_t f| \Big\|_{L^{q,\infty}} \lc \|f\|_{L^{q',1}}, \quad q= d+1-\beta. \]
This gives  the asserted restricted weak type inequality for $M_E$  at $Q_3(\beta)$.
\end{proof} 

\begin{corollary} \label{Tbeta}Let $E\subset [1,2]$ and $\dim_M\!E=\beta$.\\
(i) Then for $\frac{d-1+\beta}{d-1}<p<\infty$
$$\big\|\sup_{t\in E}| A^j_t f|\big\|_p \lc_p 2^{-ja(p)} \|f\|_p$$
with $a(p)>0$.\\
(ii) For $(1/p,1/q)$ in the interior of the triangle $\cT_\beta$ with corners $Q_1$, 
$Q_2(\beta)$, $Q_3(\beta)$ we have
$$\big\|\sup_{t\in E}| A^j_t f|\big\|_q \lc_p 2^{-ja(p,q)} \|f\|_p,$$
for some $a(p,q)>0$.
\end{corollary}
\begin{proof} 
Use
 $N(E,2^{-j}) \lc_\eps 2^{j(\beta+\eps)}$,  apply the previous lemmata to $A_t^j$.
 \end{proof}

\section{Estimates near \texorpdfstring{$Q_4(\gamma)$}{Q4}: The role of Assouad dimension}\label{assouadsect}
As the case $\beta=1$ is already known (see \cite{schlag-sogge}) we shall assume in this 
section that $\beta<1$.

Let $\gamma\le 1$ and $$p_4=\tfrac{d^2+2\ga-1}{d^2-d}, \quad q_4= \tfrac{d^2+2\ga-1}{d-1},$$ 
i.e. $Q_4(\ga)=(1/p_4, 1/q_4)$.

\begin{proposition}\label{Q4endptprop}
Let either $d\ge 3$, or both $d=2$ and $\gamma<1/2$. Suppose that assumption \eqref{endptassouadassu} holds.
Then
\begin{equation} \label{Q4rwt}
\|M_E f\|_{L^{q_4,\infty}}\lc \|f\|_{L^{p_4,1}}.
\end{equation}
\end{proposition}
\begin{proof}
Let  $\vartheta= \frac{(d-1)^2-2\gamma}{d^2+2\gamma-1}$ 
and notice that $\vartheta \in (0,1) $ if $d=3$ or $d=2$, $\gamma<1/2$. One checks that  $1-\vth= \frac{2(d-1+2\gamma)}{d^2+2\gamma-1}$ and  
$$
(\tfrac 1{p_4}, \tfrac 1{q_4}, 0)= (1-\vth) (\tfrac 12, \tfrac{d-1}{2(d-1+2\gamma)}, \tfrac{2\gamma-(d-1)^2}{2(d-1+2\gamma)} ) +\vth (1, 0, 1).
$$ For all estimates concerning $A^j_t$ we shall assume $d\ge 2$, and assumption 
\eqref{endptassouadassu}.
By Lemma \ref{pq-lemma} we have
\begin{equation} \label{1infty}
\big\|\sup_{t\in E}|A^j_t f| \big\|_\infty \lc 2^j  \|f\|_1.
\end{equation}
We shall prove, for $d\ge 2$, 
\begin{equation}\label{main2}
\big\|\sup_{t\in E}|A^j_t f| \big\|_{q_\gamma,\infty} \lc 2^{- j \frac{(d-1)^2-2\gamma}{2(d-1+2\gamma)}  } 
\|f\|_2, \quad \text{where}\;
q_\gamma=\tfrac{2(d-1+2\gamma)}{d-1}.
\end{equation}
Notice  that 
$\frac{(d-1)^2-2\gamma}{2(d-1+2\gamma)}>0$
for $d\ge 3$ or $d=2$, $\gamma<1/2$.
 The asserted 
restricted weak type inequality  follows from \eqref{1infty} and \eqref{main2}, 
using Bourgain's interpolation trick. 
It remains to prove \eqref{main2}.

For each $j$ let $\cI_j(E)$ denote the collection of intervals $J$ of the form $[k 2^{-j}, (k+1)2^{-j}]$ which intersect $E$. For each interval $I$ with length at least $2^{-j}$ we form $\cI_j(E\cap I)$. Then 
\begin{equation}\label{dyadiceff} 
\# \cI_j(E\cap I) \le 7 N(E\cap I, 2^{-j}).
\end{equation} 
Indeed if $\cV$ is any collection of intervals of length $2^{-j}$ covering $E\cap I$, and if 
$J\in \cI_j(E\cap I)$ there must be an interval $\tilde J(J)\in \cV$ which intersects $J$; moreover if $J, J'$ have distance $\ge 3\cdot 2^{-j}$ then the intervals   $\tilde J(J)$ and $\tilde J(J')$ in $\cV$ must be  disjoint.  This means that the cardinality of $\cV$ is at least one seventh of the cardinality of $\cI_j(E\cap I)$ and \eqref{dyadiceff} follows. By our assumption  \eqref{endptassouadassu} we also have
\begin{equation}\label{dyadiceffassu} 
\# \cI_j(E\cap I) \le C |I|^\gamma 2^{j\gamma}\end{equation} 
for any interval of length at least $2^{-j}$.

We now fix $j$. Let $\cI_j(E)=\{I_\nu\} $ and let $\{t_\nu\}$ be the set of left endpoints of these intervals. Here the indices $\nu$ are chosen from some finite set which we call $\Zj$. Equipping $\Zj$ with the counting measure, we claim that it suffices to show that for 
$q_\gamma= \frac{2(d-1+2\gamma)}{d-1}$,
\begin{equation}\label{eqn:q4_goal}
\|A^j_{t_\nu} f\|_{L^{q_\gamma,\infty}(\bbR^d\times \Zj)} + \int_0^{2^{-j}} \|\partial_s A_{t_\nu+s}^j f\|_{L^{q_\gamma,\infty}(\bbR^d\times \Zj)} ds \lesssim 2^{- j \frac{(d-1)^2-2\gamma}{2(d-1+2\gamma)}  } \|f\|_2.
\end{equation}
Indeed, by the fundamental theorem of calculus
\[ \sup_{t\in E} |A_t^j f| \le \sup_{\nu\in\Zj} |A_{t_\nu}^j f| + \int_0^{2^{-j}} \sup_{\nu\in\Zj} |\partial_s A_{t_\nu+s}^j f| ds \]
Taking $L^{q_\gamma,\infty}$-norms (recall that $L^{q,\infty}$ is normable, see \cite{hunt}) on both sides and noting that
\[ \meas(\{ x\,:\,\sup_{\nu\in\Zj} |g(x,\nu)|>\lambda \}) \le \meas_{\bbR^d\times \Zj}(\{ (x,\nu)\,:\, |g(x,\nu)|>\lambda \}) \]
we see that $\|\sup_{t\in E} |A_t^j f|\|_{q,\infty}$ is dominated by a constant times the left hand side of \eqref{eqn:q4_goal}.

The estimate \eqref{eqn:q4_goal} follows once we show that
\begin{equation}\label{main23}
2^{-j(d-1)/2}
\|T^j_{t_\nu} f \|_{L^{q_\gamma,\infty}(\bbR^d\times \Zj)} \lc 2^{- j \frac{(d-1)^2-2\gamma}{2(d-1+2\gamma)}  } \|f\|_2.
\end{equation}
Given a function $g:\bbR^d\times \Zj\to \bbC$, define the operator
\[ S_j g (x,\nu) = 2^{-j(d-1)} \sum_{\nu'\in\Zj} T^j_{t_{\nu}} (T^j_{t_{\nu'}})^* 
[g(\cdot,\nu')](x). \]
A $TT^*$ argument using that the dual space of $L^{q',1}$ is $L^{q,\infty}$ shows that \eqref{main23} follows once we establish
\begin{equation}\label{TT*}
\| S_j g \|_{L^{q_\gamma,\infty}(\bbR^d\times \Zj)} \lesssim 2^{- j \frac{(d-1)^2-2\gamma}{d-1+2\gamma}  } \|g\|_{L^{q_\gamma',1}(\bbR^d\times \Zj)}.
\end{equation}

We use a variant of the  argument in the proof of the $L^2$ Fourier restriction theorem  \cite{tomas} (see also \cite{strichartz}). For $n\ge 0$ and $\nu\in\Zj$ we define 
\[\cZ_{n,j}(\nu):=\{\nu'\in\Zj\,:\,2^{-j+n-1}\le |t_\nu-t_{\nu'}|<2^{-j+n}\}.\]
Observe that $\cZ_{n,j}(\nu)$ is empty if $n\ge j+3$ and that $\Zj = \bigcup_{n\ge 0} \cZ_{n,j}(\nu)$.
Define the operators $S_{n,j}$ acting on functions $g:\bbR^d\times\Zj\to \bbC$ by
\[ S_{n,j} g(x,\nu) = 2^{-j(d-1)} \sum_{\nu'\in \cZ_{n,j}(\nu) } T^j_{t_{\nu}} (T^j_{t_{\nu'}})^* 
[g(\cdot,\nu')] (x). \]
Then $S_j = \sum_{n\ge 0} S_{n,j}$. We claim that
\begin{equation}\label{eqn:Q4-Linfmain}
\|S_{n,j} g\|_{L^{\infty}(\bbR^d\times \Zj)} \lesssim 2^{-n(d-1)/2+j} \|g\|_{L^1(\bbR^d\times \Zj)}
\end{equation}
and
\begin{equation}\label{eqn:Q4-L2main}
\|S_{n,j} g\|_{L^{2}(\bbR^d\times \Zj)} \lesssim 2^{n\gamma-j(d-1)} \|g\|_{L^2(\bbR^d\times \Zj)}.
\end{equation}
Then \eqref{TT*} follows by Bourgain's interpolation trick: with $\theta~=~2/q_\gamma~=~\tfrac{d-1}{d-1+2\gamma}$, 
\[(\tfrac1{q_\gamma'}, \tfrac1{q_\gamma}, 0) = \theta (\tfrac12, \tfrac12, \gamma) + (1-\theta) (1,0,-\tfrac{d-1}2).\]
From Lemma \ref{ptwlem} we get that the convolution kernel $K^j_{\nu,\nu'}$ of  $T^j_{t_{\nu}} (T^j_{t_{\nu'}})^* $ satisfies
\begin{equation}\label{eqn:TTs_kerest}
\|K^j_{\nu,\nu'} \|_\infty\lc 2^{jd} (1+ 2^j|t_\nu-t_{\nu'}|)^{-\frac{d-1}{2}}.
\end{equation}
This implies \eqref{eqn:Q4-Linfmain}. It remains to prove \eqref{eqn:Q4-L2main}. Using
the Cauchy-Schwarz inequality we get
\begin{align*}
&\Big(\sum_{\nu\in\Zj} \Big\|\sum_{\nu' \in \cZ_{n,j}(\nu)}
T^j_{t_{\nu}} (T^j_{t_{\nu'}})^* 
[g(\cdot,\nu')]\Big\|_2^2\Big)^{1/2}
\\
&\quad\le \Big(\sum_{\nu\in\Zj} \#(\cZ_{n,j}(\nu))
\sum_{\nu' \in \cZ_{n,j}(\nu)}
\Big\|
T^j_{t_{\nu}} (T^j_{t_{\nu'}})^* 
[g(\cdot,\nu')]\Big\|_2^2\Big)^{1/2}
\\
&\quad\le \Big(\sum_{\nu\in\Zj} \#(\cZ_{n,j}(\nu))
\sum_{\nu' \in \cZ_{n,j}(\nu)}
\big\|
g(\cdot,\nu')\big\|_2^2\Big)^{1/2},
\end{align*}
where we have used that $\|T^j_t\|_{L^2\to L^2} = O(1)$. Finally, by \eqref{dyadiceffassu} we have $\#\cZ_{n,j}(\nu)\lc2^{n\gamma}$ for all $\nu\in\Zj$. Together with the previous display this implies \eqref{eqn:Q4-L2main}.
\end{proof}
  
  The above proof also gives
  \begin{corollary} \label{Q4propeps}
  Suppose that $\dim_A\!E=\gamma$. Then for all $\eps>0$ 
\begin{equation}\label{Q4j}\big\|\sup_{t\in E} |A^j_t f|\big\|_{q_4} \lc_\eps 2^{j\eps} \|f\|_{p_4}.
\end{equation}
\end{corollary}
  
  \begin{proof} The assumption means that given any $\eps>0$ the assumption \eqref{endptassouadassu} holds with $\gamma+\eps$ in place of $\gamma$. Hence we get 
  \eqref{main2} with an additional factor of $C(\ep)2^{j\ep} $ for all $\ep>0$, and interpolation as before yields the result.  
  \end{proof}
  
  \begin{proof}[Proof of Theorems \ref{pqthm} and \ref{pqthmendpt}]
  Theorem \ref{pqthm} is now immediate from Corollary \ref{Tbeta} and Corollary \ref{Q4propeps}. Theorem \ref{pqthmendpt} follows by a combination of 
  Lemma \ref{pqrwtlemma}, Proposition \ref{Q4endptprop} and real interpolation.  
  \end{proof}
  
\section{Necessary conditions: Proof of Theorem \ref{thm:lowerbd}} \label{sec:lowerbd}
Let $\beta=\dim_M\!E$ and suppose that $\theta\in[0,1)$ is such that $\dimthE = \tfrac{\beta}{1-\theta}.$ Set $\widetilde{\gamma}=\tfrac{\beta}{1-\theta}$ and assume that $(1/p,1/q)$ is such that $M_E$ is bounded from $L^p(\bbR^d)$ to $L^q(\bbR^d)$. We will show that $(1/p,1/q)\in \mathcal{Q}(\beta,\widetilde{\gamma})$.

This is done by providing four separate examples, each corresponding to one of the (generically) four edges of $\mathcal{Q}(\beta,\widetilde{\gamma})$.
One is just in view of translation invariance \cite{hormander}, and two others 
are adaptations of standard examples for spherical means and maximal functions (see \cite{schlag}, \cite{schlag-sogge}, \cite{sww1}). The last example reveals the role of the Assouad spectrum.

\subsection{The line connecting \texorpdfstring{$Q_1$ and $Q_2(\beta)$}{Q1 and Q2}} \label{sec:lowerbd1}
This is simply the necessary condition $p\le q$ imposed by translation invariance on $\bbR^d$;
one tests $M_E$ on $f+ f(\cdot-a)$ where $f$ is compactly supported and $a$ is a large vector,  see \cite{hormander}.

\subsection{The line connecting \texorpdfstring{$Q_2(\beta)$ and $Q_3(\beta)$}{Q2 and Q3}} \label{sec:lowerbd2} First let $B_\delta$ be  the ball of radius $\delta\ll 1$ centered at the origin and $\chi_\delta$ the characteristic function 
of $B_\delta$, so that $\|f\chi_\delta\|_p\le \delta^{d/p}$. The maximal function $M_E$ is of size $\gc \delta^{d-1}$ on a union of annuli with measure $N(E,\delta)\delta$.  This leads to the inequality
$$\delta^{d-1+1/q}N(E, \delta)^{1/q}\lc \delta^{d/p}.$$
By the assumption $\dim_{M}\!E=\beta$ we have given $\eps>0$ a sequence $\delta_m$, with $\delta_m\to 0$ as $m\to \infty$, such that
$N(E,\delta_m) \ge \delta_m^{\eps-\beta}$. Hence, after letting $\eps\to 0$  we get the condition 
\begin{equation} \tfrac{1-\beta}q+ d-1\ge \tfrac dp\,\end{equation}
as being necessary for $L^p\to L^q$ boundedness. 

\subsection{The line connecting \texorpdfstring{$Q_1$ and $Q_4(\widetilde{\gamma})$}{Q1 and Q4}} \label{sec:lowerbd3}
As in \cite{schlag} we may take $f_\delta= \bbone_{\cC(\delta,t)}$ where $\cC(\delta,t)$ is the $\delta$ neighborhood of the  circle of radius $t\in [1,2]$ centered at the origin. Then $\|f_\delta\|_p=\delta^{1/p}$ and $|A_t f(x)| \ge 1$ for  $|x|\le c \delta$. Hence we $\delta^{d/q}\lc \delta^{ 1/p}$ which forces $d/q\ge 1/p$, as required.

\subsection{The line connecting \texorpdfstring{$Q_3(\beta)$ and $Q_4(\widetilde{\gamma})$}{Q3 and Q4}} \label{sec:lowerbd4}

\def\thegamma{\widetilde{\gamma}}
By assumption, for every $\varepsilon>0$ there exists an arbitrarily small $\delta>0$ and an interval $I\subset [1,2]$ with $|I|=\delta^\theta$ such that $N(E\cap I,\delta)\ge (|I|/\delta)^{\widetilde{\gamma}-\varepsilon}$.
Set $\alpha=\beta/\thegamma$ and \[ \sigma=\delta^{ \alpha/2}\ge \delta^{1/2}.\]
Let $r$ be the left endpoint of the interval $I$ and let $g_{\delta,I}$ be the characteristic function of the set
\[ \{ (y',y_d)\,:\, ||y|-r|\le \delta,\,|y'|\le \sigma \}. \]
Then 
\[ \|g_{\delta,I}\|_p \approx (\delta \sigma^{d-1})^{1/p} = \delta^{(1+\frac{\alpha}2(d-1))\frac1p}. \]
Choose a covering of $E\cap I$ by a collection $\mathcal{J}$ of pairwise disjoint intervals, each of length $\delta$ such that $E\cap I\cap J\not=\emptyset$ for every $J\in\mathcal{J}$. Then $\# \mathcal{J}\ge N(E\cap I,\delta)$. 

Let $c\in (0,1)$ be a sufficiently small absolute constant not depending on dimension that is to be determined. We claim that for all $t\in \cup_{J\in \cJ}J $ and  all $x=(x',x_d)$ with $|x'|\le c \delta\sigma^{-1}$ and 
$|x_d+t-r|\le c \delta$,
\begin{equation}\label{eqn:counterex_lowerbd}
M_E g_{\delta,I}(x) \ge A_t g_{\delta,I}(x',x_d) \gtrsim \sigma^{d-1}.
\end{equation}
Indeed, let $y=(y',y_d)\in S^{d-1}$ with $|y'|\le c \sigma$. Compute
\begin{align*}
|x+ty|^2 &= |x'|^2 + x_d^2 + 2t \inn{x'}{y'} + 2tx_d y_d + t^2\\
& = |x'|^2 + (x_d+t)^2 + 2tx_d (\sqrt{1-|y'|^2}-1) + 2t \inn{x'}{y'}.
\end{align*}
Since $|x_d + t - r|\le c \delta$ and $|x'|^2 \le c^2 \delta^2\sigma^{-2}\le c^2 \delta$,
\[ ||x'|^2 + (x_d+t)^2 - r^2| \le 6 c \delta, \]
\[ |2t\inn{x'}{y'}| \le 4 |x'| |y'| \le 4 c^2 \delta, \]
\[ |2t x_d (\sqrt{1-|y'|^2}-1)| \le 2 (|t-r|+c\delta) |y'|^2 \le 2c (|I|+c\delta) \sigma^2 \le 4 c \delta,\]
where we used that $|I|=\delta^\theta=\delta \sigma^{-2}$. This implies 
\[ ||x+ty|^2 - r^2| \le 14 c \delta, \]
and hence $||x+ty|-r|\le \delta$ when we pick $c$ small enough (say, $c=10^{-2}$).
Also, $|x'+ty'| \le |x'| + 2 |y'| \le \sigma$ so that altogether we proved $g_{\delta,I}(x+ty) = 1$. This establishes \eqref{eqn:counterex_lowerbd}.
Since the intervals $J\in\mathcal{J}$ are disjoint, the corresponding regions of $x$ where \eqref{eqn:counterex_lowerbd} holds can be chosen disjoint. Hence,
\[ \|M_E g_{\delta,I} \|_q \gtrsim \sigma^{d-1} (N(E\cap I,\delta) \delta \cdot (\delta\sigma^{-1})^{d-1})^{1/q} \]
Finally, we estimate $N(E\cap I,\delta)\ge (|I|/\delta)^{\widetilde{\gamma}-\varepsilon}=\delta^{-\beta+\varepsilon\alpha}$and let $\varepsilon$ and $\delta$ tend to zero to find the necessary condition
\[ L(\tfrac1p,\tfrac1q):=\tfrac{\alpha}2(d-1) + (d-\beta - (d-1)\tfrac\alpha{2})\tfrac1q - (1+\tfrac\alpha{2}(d-1))\tfrac1p \ge 0. \]
A  computation shows that $L(Q_3(\beta))=0$ and $L(Q_4(\thegamma))= L(Q_4(\beta/\alpha))=0$.

\section{Examples of Assouad regular sets} \label{Assreg}

Let $0<\beta<\gamma<1$. We  construct a $(\beta,\gamma)$-Assouad regular subset of $[1,2]$. In what follows we put $\lambda = 2^{-1/\beta}$ and $\mu=2^{-1/\gamma}$,
so that $\lambda<\mu<1/2$.

\subsection{Cantor set construction} We review the  standard Cantor set construction  adapted to a compact interval $I_{0,1}=[a,b]$, see \cite[p. 60]{mattila}.
We let $I_{1,1}^\mu$ be the compact interval of length $\mu(b-a)$ that includes the left endpoint of $I_{0,0}^\mu$ and let $I_{1,2}^\mu$ be the compact interval of length $\mu(b-a)$ that includes the right endpoint of $I_{0,0}$.
Continue  this selection for the two compact subintervals.
At stage $k-1$ we get $2^{k-1}$ intervals 
$I_{k-1,1}^\mu,\dots, I_{k-1, 2^{k-1}}^\mu$ of length $\mu^k(b-a)$. 

We let $C^\mu_k([a,b])= \cup_{\nu=1}^{2^k} I_{k,\nu}^\mu$ 
and let  $\bd(C^\mu_k([a,b]))$  the set of boundary points of the $2^k$ intervals $I_{k,1}^\mu,\dots I_{k,2^k}^\mu$. The usual  Cantor set is given by $C^\mu([a,b])=\cap_{k=1}^\infty C^\mu_k([a,b])$; it is of Hausdorff dimension and Assouad dimension $\gamma$.
However in our example below we will not work with the full Cantor sets.

\subsection{Construction of the set \texorpdfstring{$E$}{E}} \label{Econstr}
Let $J_k=[1+ \la^{k+1}, 1+\la^k]$. 
We now start to build a Cantor set with dissection $\mu=2^{-1/\gamma}$ on each  interval $J_k$, however to keep the Minkowski dimension $\beta$ we shall, for a suitable integer $m(k)$, stop at the $m(k)^{\text{th}}$ generation and only take the endpoints of the $2^{m(k)}$ resulting intervals of length \begin{equation}\label{deltak} \delta_k:=\la^{k-1}  \mu^{m(k)}
= 2^{-k/\beta-m(k)/\gamma}. \end{equation}  
Let $\theta=1-\beta/\gamma\in (0,1)$. Then we set $m(k)= 1+\lfloor \tfrac{k}{\theta}\rfloor.$
This choice is made so that
\begin{equation}\label{Jklen}
\delta_k^\theta\approx |J_k| \approx 2^{-k/\beta}.
\end{equation}
We then  set
$$E=  \bigcup_{k=1}^\infty E_k \text{ where } E_k=\bd(C_{m(k)}^\mu(J_k)).$$

\subsection{Dimensional estimates} 

\begin{lemma} 
For $0<\beta<\gamma<1$, $\theta=1-\beta/\gamma$ and $E$ as constructed in \S\ref{Econstr} we have that
\[ \dim_A\!E =\gamma, \quad  \dim_{A,\theta}\! E= \gamma, \quad \dim_M\!E=\beta.
\] More precisely, the quantities
\begin{equation*}
\begin{array}{*5{>{\displaystyle}l}}
(i) \qquad&\varlimsup_{\delta\to 0} \delta^\beta N(E,\delta), &(ii) &\varliminf_{\delta\to 0} \delta^\beta N(E,\delta),\\

(iii) &\varlimsup_{\delta\to 0}\sup_{|I|=\delta^\theta} 
\big(\tfrac\delta{|I|}\big)^{\gamma} N(E\cap I,\delta), &(iv) &\varliminf_{\delta\to 0}\sup_{|I|=\delta^\theta} 
\big(\tfrac\delta{|I|}\big)^{\gamma} N(E\cap I,\delta),\\ 

(v) &\varlimsup_{\delta\to 0} \sup_{\delta\le |I|} \big(\tfrac\delta{|I|}\big)^{\gamma} N(E\cap I,\delta), &(vi) &\varliminf_{\delta\to 0} \sup_{\delta\le |I|} \big(\tfrac\delta{|I|}\big)^{\gamma} N(E\cap I,\delta).
\end{array}
\end{equation*}
are all finite and positive.

\end{lemma}
\begin{proof}
Let us first show
\begin{equation} \label{assouadgammaeq}
\dim_A\!E= \gamma.\end{equation}
In order to see that 
$\dim_A\!E\le \gamma$ note that, in view of the Cantor structure of each $E_k$ with dissection $\mu=2^{-1/\gamma}$,  we get for $I\subset J_k$
\begin{equation}\label{NEkI}  N(E_k\cap I,\delta)\lc 
\begin{cases} (\delta/|I|)^{-\gamma} &\text{ if } \delta_k<\delta< |I|\le|J_k|,
\\
(\delta_k/|I|)^{-\gamma} &\text{ if } \delta<\delta_k\le |I|\le|J_k|.
\end{cases} 
\end{equation}
Then, for an arbitrary interval $I\subset [1,2]$ and $\delta<|I|$,
\begin{align*} N&(E\cap I,\delta) \le \sum_{k\ge 0} N(E_k\cap (J_k\cap I),\delta) 
\\
&\lesssim \delta^{-\gamma} \big(\sum_{\substack{k\,:\,1\ge 2^{-k/\beta}>|I|
\\ J_k\cap I\neq \emptyset}} |I|^\gamma + \sum_{k\,:\,2^{-k/\beta}\le |I|} 2^{-k{\gamma}/{\beta}} \big) \lesssim 
\delta^{-\gamma} |I|^\gamma.
\end{align*} This gives $\dim_A\!E\le \gamma$, and also shows that the quantities  
(iii), (iv), (v), (vi) 
are finite. We can also conclude  $\dim_{A,\theta}\!E \le \gamma.$ 

Next we observe,
\begin{equation}\label{assouadlower}N(E\cap J_k,\delta_k)
= N(E_k, \delta_k) = 2^{m(k)} \approx (\delta_k/|J_k|)^{-\gamma}
\end{equation}
This also shows that the quantity  (iii) is positive (also  (iv), (v), (vi))  
and that $\dim_A\!E\ge  \gamma$. Thus  
we have now proved \eqref{assouadgammaeq}.

Note that we have not yet made use of the particular choice of $m(k)$ (that is, \eqref{Jklen}).
Taking \eqref{Jklen} 
into account we see that \eqref{assouadlower} also implies that
the quantity  (iv) is positive (hence also the ones in (iii), (v), (vi)) and that 
$\dim_{A,\theta}\!E \ge \gamma.$ 
Moreover using \eqref{Jklen} we also obtain  \[N(E\cap J_k,\delta_k)\approx \delta_k^{-\beta}\] which implies the positivity of  (ii) (and (i)),  and 
$\dim_M\!E \ge \beta.$

It now only remains to consider the upper bounds for $N(E, \delta)$.
These again depend on \eqref{Jklen}. Let $\delta \in (0,1)$ be given.
Since $\delta_k^\theta \approx |J_k|$,
\[ N\Big(\bigcup_{k\,:\,\delta\ge \delta_k^\theta} E_k, \delta \Big) \lesssim 1. \]
This gives
\[ N(E,\delta) \lesssim 1 + \sum_{k: \delta_k<\delta<\delta_k^\theta} N(E_k,\delta) + \sum_{k\ge 0: \delta_k\ge \delta} N(E_k,\delta), \]
which by \eqref{NEkI} (with $I=J_k$) and \eqref{Jklen} is
\[ \lesssim \sum_{k\,:\,\delta_k<\delta<\delta_k^{\theta}} \delta^{-\gamma} \delta_k^{\gamma-\beta} + \sum_{k\ge 0: \delta_k\ge \delta} \delta_k^{-\gamma} \delta_k^{\gamma-\beta} \lesssim \delta^{-\beta}. \]
Hence we proved the finiteness of the quantities (i), (ii) and the bound
$\dim_M\!E \le \beta.$
\end{proof}

\section{A consequence for sparse domination bounds}\label{sparsesect}
One motivation to prove sharp  $L^p \to L^q$ estimates comes from the problem of sharp  
sparse domination bounds for 
the global maximal operator
$$\fM_E f(x)= \sup_{k\in \bbZ}\sup_{t\in E} |A_{2^k t}f(x)|,$$
as suggested in \S7.5.3 in \cite{lacey}, with various consequences to weighted norm inequalities. 
The concept of sparse domination  originates in Lerner's paper \cite{lerner1}. Here we use the definition of sparse domination of bilinear forms in \cite{lacey}, which in some form 
goes back to \cite{bernicotetal}.
We refer the reader to   \cite{ccdo}, \cite{lacey} for many additional references and historical remarks.

A collection $\cS$ of cubes is called {\it sparse} 
 if for every $Q\in \cS$,  there is a measurable set $A_Q\subset Q$ so that $|A_Q|\ge |Q|/4$ such that the sets $\{A_Q: Q\in \cS\}$ are disjoint. 

\medskip




\noi{\it Definition.} Let  $(p_1, p_2)$ be a pair of exponents, each in $[1,\infty)$.
Let $T$ be a sublinear operator $T$ mapping compactly supported $L^{p_1}$   functions in $\bbR^n$ to locally integrable functions on $\bbR^n$. 
For a sparse family $\cS$ we set 
\[\Lambda_{\cS, p_1, p_2}(f,g):= 
\sum_{Q\in \cS} |Q| \Big(\frac{1}{|Q|} \int_Q |f(x)|^{p_1} dx\Big)^{1/p_1}
\Big(\frac{1}{|Q|} 
\int_{Q} |g(x)|^{p_2} dx\Big)^{1/p_2} .
\] Then 
is called the  sparse form associated with $\cS$.
We say  that $T$ satisfies a {\it $(p_1,p_2)$ sparse domination inequality}  if there is a constant $C$  such that
\begin{equation}\label{sparsedef} 
\Big|\int Tf(x) g(x) dx\Big| \le \\ C 
\sup\, \{ \Lambda_{\cS,p_1,p_2}(f,g):\,\cS  \text{ sparse} \}
\end{equation}
holds for all continuous compactly supported $f$ and locally integrable $g$; here the supremum is taken over all sparse families $\cS$.
We define 
$\|T\|_{{\rm sp}(p_1,p_2)} $ 
as the infimum over all $C>0$ such that \eqref{sparsedef} holds for all $f\in L^{p_1}$, $g\in L^{p_2}$ with compact support.
It is easy to see that $\|T\|_{L^p\to L^p} \lc \|T\|_{{\rm sp}(p_1,p_2)}$,  for $p_1<p<p_2'$;  see e.g. \cite[Prop. 6.1]{lacey}.

\medskip 
Let $E\subset [1,2]$ and consider the global maximal function  
\[\fM_E f(x)= \sup_{k\in \bbZ} \sup_{t\in E} |A_{2^k t}f(x)|\] 
mentioned in the introduction.
The paper by Lacey \cite{lacey} shows that Theorem \ref{pqthm} and a related  regularity result imply certain sparse domination inequalities for the $\fM_E$ mentioned.
Lacey's result   covered  the cases $E=\{\text{point}\}$ and $E=[1,2]$. For general $E\subset [1,2]$ we get

\begin{thms} \label{sparsethm} Let $0\le \beta \le \gamma\le 1$, $d\ge 3$ or $0\le \beta\le \gamma \le 1/2$, $d=2$. Let  $E$ be as in Theorem \ref{pqthm}.
Suppose that $(p_1^{-1}, 1-p_2^{-1})$ belongs to the interior  of $\cR(\beta,\gamma)$. Then
$$\|\fM_E\|_{{\rm sp}(p_1,p_2)} <\infty.$$
\end{thms}

The needed regularity result alluded to above  is 
\begin{lemma} \label{regprop}  Let $E$ be as in Theorem \ref{pqthm}. Then  for $(1/p,1/q)\in \cR(\beta,\gamma)$ there is $\alpha(p,q)>0$ such that
\begin{equation}\label{regpq}\|\sup_{t\in E} |A_t f(\cdot+h)-A_t f(\cdot) | \|_q \lc  |h|^{\alpha(p,q)} \|f\|_p.
\end{equation}
\end{lemma}

\begin{proof} This regularity result is of course a by-product of the proof of Theorem \ref{pqthm}.
We have, for $A_t^j f$ as in \eqref{Ajtdef},
\begin{align*}
\big\|\sup_{t\in E} |A_t^j f|\big \|_q 
+2^{-j} \Big\|\sup_{t\in E} \big|\nabla_{x,t} A_t^j f\big|\Big \|_q 
&\lc  2^{-j\eps(p,q)} \|f\|_p, 
\end{align*}
for $\eps(p,q)>0$ if 
$(1/p,1/q)\in \cR(\beta,\gamma)$. This immediately implies \eqref{regpq}, for  some $\alpha(p,q)>0$.
\end{proof} 

\begin{proof}[Proof of Theorem \ref{sparsethm}]
The reduction in \cite{ccdo}, \cite{lacey}  can be applied
(see also 
\cite{roberlin} for related arguments). One systematically replaces in \cite{lacey}
the full local maximal operator  $M_{[1,2]}$
 by its   modification $M_E$  for general $E\subset [1,2]$ and uses  Theorem \ref{pqthm} and Lemma \ref{regprop} in the proof.
 \end{proof} 

\begin{remark} If in this  proof one uses the $L^p(\bbR^2)\to L^q(\bbR^2)$ result in \cite{rs} one can drop the condition $\gamma\le 1/2$ in the two-dimensional case of Theorem \ref{sparsethm}.
\end{remark} 
\begin{remark} One can also  obtain  sparse domination results for the general  spherical maximal operator $$\cM_E f = \sup_{t\in E} |A_t f|$$ when $E\subset (0,\infty)$. In this context, one has to use   dilation invariant notions of the Minkowski and Assouad dimensions for the sets $E \cap[\la,2\lambda]$, with uniformity in $\lambda$ in the definitions. Specifically, if
$E_\lambda: = \la^{-1} E\cap[1,2]$
we then let $\beta $ be the infimum over all $\widetilde \beta>0$  for which
\[\sup_{\la>0} \sup_{\delta\in (0,1)} \delta^{\widetilde \beta} N(E_\lambda, \delta) <\infty.\]
We let $\gamma$ be the infimum over all $\widetilde \gamma>0$  for which 
\[\sup_{\la>0} \sup_{I\subset [1,2]} \sup_{\delta\in (0,1)} ( \delta/{|I|})^{\widetilde \gamma} N(E_\lambda, \delta) <\infty.\]  
 Then 
$\|\cM_E\|_{{\rm sp}(p_1,p_2)} <\infty$  holds under the assumption that $(p_1^{-1}, 1-p_2^{-1})$ belongs to $\cR(\beta, \gamma)$.
\end{remark}

\newpage

\begin{thebibliography}{99}
\bibitem{bernicotetal} Fr\'{e}d\'{e}ric Bernicot, Dorothee Frey, Stefanie Petermichl. \emph{Sharp weighted norm estimates beyond Calder\'on-Zygmund theory.} \arxiv{1510.00973}. Anal. PDE 9 (2016), no. 5, 1079--1113.	
	
\bibitem{bourgain1} 
Jean Bourgain. \emph{Estimations de certaines fonctions maximales.} C. R. Acad. Sci. Paris S\'er. I 301 (1985), 499--502.
\bibitem{bourgain2} Jean Bourgain. \emph{Averages in the plane over convex curves and maximal operators.} J. Anal. Math. 47 (1986), 69--85. 

\bibitem{carberyetal}
Anthony Carbery, Andreas Seeger, Stephen  Wainger,  James Wright. \emph{Classes of singular integral operators along variable lines.} 
J. Geom. Anal. 9 (1999), no. 4, 583--605. 

\bibitem{ccdo} Jos\'e M. Conde-Alonso, Amalia Culiuc, Francesco  Di Plinio, Yumeng  Ou. \emph{A sparse domination principle for rough singular integrals.} \arxiv{1612.09201}. Anal. PDE 10 (2017), no. 5, 1255--1284.

\bibitem{dv98} Javier Duoandikoetxea, Ana Vargas. \emph{Maximal operators associated to Fourier multipliers with an arbitrary set of parameters.} Proc. Royal Soc. Edinburgh A 128 (1998), no. 4, 683--696. 

\bibitem{fraser2014}
Jonathan M. Fraser. \emph{Assouad type dimensions and homogeneity of fractals.} \arxiv{1301.2934}. Trans. Amer. Math. Soc. 366 (2014), no. 12, 6687--6733. 

\bibitem{fraser-yu1}
Jonathan M. Fraser, Han Yu. \emph{New dimension spectra: finer information on scaling and homogeneity.} \arxiv{1610.02334}. Adv. Math. 329 (2018), 273--328. 

\bibitem{fraser-yu2}Jonathan M. Fraser, Han Yu. \emph{Assouad-type spectra for some fractal families.} \arxiv{1611.08857}. Indiana Univ. Math. J. 67 (2018), no. 5, 2005--2043.  

\bibitem{hormander} Lars H\"ormander.  \emph{Estimates for translation invariant operators in $L^p$  spaces.} Acta Math. 104 (1960), 93--140. 

\bibitem{hunt}
Richard A. Hunt. \emph{On $L(p,q)$ spaces.} Enseign. Math. (2) 12 (1966), 249--276.

\bibitem{lacey} Michael T. Lacey. 
\emph{Sparse bounds for spherical maximal functions.} 
\arxiv{1702.08594v6}.
J. Anal. Math. 139 (2019), no. 2, 613--635. 


\bibitem{leckband} Mark A. Leckband. \emph{A note on the spherical maximal operator for radial functions.} Proc. Amer. Math. Soc. 100 (1987), 635--640. 

\bibitem{slee} Sanghyuk Lee. \emph{Endpoint estimates for the circular maximal function.} Proc. Amer. Math. Soc. 131 (2003), no. 5, 1433--1442. 

\bibitem{lerner1}
Andrei  K. Lerner. \emph{A simple proof of the  $A_2$ conjecture.} \arxiv{1202.2824}. Int. Math. Res. Not. (2013), no. 14, 3159--3170.

\bibitem{littman} 
 Walter Littman. \emph{ $L^p-L^q$-estimates for singular integral operators arising from hyperbolic equations.} Partial differential equations (Proc. Sympos. Pure Math., Vol. XXIII, Univ. California, Berkeley, Calif., 1971), pp. 479--481. Amer. Math. Soc., Providence, R.I., 1973.
 
\bibitem{mattila} Pertti Mattila. Geometry of sets and measures in Euclidean spaces. Fractals and Rectifiability.
Cambridge Studies in Advanced Mathematics. Cambridge University Press 1995.

\bibitem{roberlin}
Richard Oberlin. \emph{Sparse bounds for a prototypical singular Radon transform.} \arxiv{1704.04297}. Canad. Math. Bull. 62 (2019), no. 2, 405--415. 


\bibitem{rs} Joris Roos, Andreas Seeger.
\emph{Spherical  maximal functions  and fractal dimensions of dilation sets.} \arxiv{2004.00984}, 2020.

\bibitem{schlag} Wilhelm Schlag. \emph{A generalization of Bourgain's circular maximal theorem.} 
J. Amer. Math. Soc. 10 (1997), no. 1, 103--122. 

\bibitem{schlag-sogge}
Wilhelm  Schlag, Christopher D.  Sogge.  \emph{Local smoothing estimates related to the circular maximal theorem.} Math. Res. Lett. 4 (1997), no. 1, 1--15. 

\bibitem{stw} Andreas  Seeger,  Terence  Tao,  James  Wright. \emph{Endpoint mapping properties of spherical maximal operators.} \arxiv{math/0205153}. J. Inst. Math. Jussieu 2 (2003), no. 1, 109--144. 

\bibitem{sww1} 
Andreas Seeger, Stephen  Wainger, James Wright. \emph{Pointwise convergence of spherical means.} \arxiv{math/0205154}. Math. Proc. Camb. Phil. Soc. 118 (1995), 115--124. 
\bibitem{sww2} Andreas Seeger, Stephen  Wainger, James Wright. \emph{
Spherical maximal operators on radial functions.} \arxiv{math/9601220}. Math. Nachr. 187 (1997), 95--105.

\bibitem{stein} Elias M. Stein. \emph{Maximal functions: spherical means.} Proc. Natl. Acad. Sci. USA 73 (1976), 2174--2175. 

\bibitem{Steinbook3}
Elias M. Stein. 
{Harmonic Analysis: real-variable methods,
                  orthogonality, and oscillatory integrals.}
With the assistance of Timothy S. Murphy.
Princeton University Press, Princeton, NJ, 1993.

\bibitem{strichartz}Robert Strichartz. \emph{Restrictions of Fourier transforms to quadratic surfaces and decay of solutions of wave equations.} Duke Math. J. 44 (1977), 705--714.

\bibitem{tomas} Peter Tomas. \emph{A restriction theorem for the Fourier transform.} Bull. Amer. Math. Soc. 81 (1975), 477--478. 

\end{thebibliography}
\end{document}